\documentclass[pdflatex,sn-mathphys-num]{sn-jnl}

\usepackage{graphicx}%
\usepackage{multirow}%
\usepackage{amsmath,amssymb,amsfonts}%
\usepackage{amsthm}%
\usepackage{mathrsfs}%
\usepackage[title]{appendix}%
\usepackage{xcolor}%
\usepackage{textcomp}%
\usepackage{manyfoot}%
\usepackage{booktabs}%
\usepackage{algorithm}%
\usepackage{algorithmicx}%
\usepackage{algpseudocode}%
\usepackage{listings}%

\theoremstyle{thmstyleone}%
\newtheorem{theorem}{Theorem}[section]
\newtheorem{lemma}[theorem]{Lemma}

\theoremstyle{thmstyletwo}%
\newtheorem{example}[theorem]{Example}%

\theoremstyle{thmstylethree}%

\numberwithin{equation}{section}
\numberwithin{table}{section}

\begin{document}

\title[External Difference Families Arising from
Two or Three Cyclotomic Classes]{External Difference Families Arising from
Two or Three Cyclotomic Classes}


\author*[1]{\fnm{Miwako} \sur{Mishima}}\email{mishima.miwako.n0@f.gifu-u.ac.jp}

\author[1]{\fnm{Yu} \sur{Tsunoda}}\email{tsunoda.yu.z6@f.gifu-u.ac.jp}

\affil[1]{\orgdiv{Department of Electrical, Electronic and Computer Engineering}, \orgname{\\Gifu University}, \orgaddress{\street{1-1 Yanagido}, \city{Gifu}, \postcode{501-1193}, 
\country{Japan}}}



\abstract{
We study external difference families arising from cyclotomic classes in finite fields from the viewpoint of a fixed number of blocks. 
For families consisting of even-indexed cyclotomic classes, the EDF condition can be expressed in terms of relations among cyclotomic numbers. 
We first study the two-block case and recover a classical characterization in terms of quadratic forms.
Our main result shows that, for a prime \(p=12k+1\), the family
\(\{C_0^6,C_2^6,C_4^6\}\)
forms an EDF in \(\mathbb{F}_p\) if and only if \(k\) is a square.
The proofs combine symmetry relations of cyclotomic numbers with their explicit evaluations.
}




\maketitle

\section{Introduction}\label{sec1}

External difference families (EDFs) and related combinatorial structures play important roles in algebraic manipulation detection (AMD) codes and  synchronization systems (see, for example, \cite{paterson2016combinatorial,mutoh2008difference}). 
Cyclotomic constructions over finite fields provide a standard approach  to constructing EDFs and related difference systems.

Let \(q=ef+1\) be a prime power and let \(C_i^e\) denote the \(i\)-th cyclotomic class of order \(e\) in the finite field \(\mathbb{F}_q\). 
By selecting suitable cyclotomic classes or unions of such classes, one obtains many classical examples of EDFs and related structures. Constructions of this type have been studied extensively from the viewpoints of difference systems of sets, authentication codes, and AMD codes. 
Recently, Huczynska and Johnson \cite{huczynska2023internal} introduced the notion of an external partial difference family (EPDF), which generalizes EDFs by allowing two distinct multiplicities for external differences. 
Cyclotomic constructions provide natural examples of EPDFs, and the problem of determining when such constructions form EDFs has also been studied.

In this paper, we focus on cyclotomic constructions with a fixed number of blocks. More precisely, we study families consisting of prescribed cyclotomic classes and determine when they form EDFs. 
This viewpoint leads naturally to explicit arithmetic conditions involving cyclotomic numbers. 
The two-block case is closely related to classical cyclotomic constructions, whereas the three-block case leads to substantially richer arithmetic conditions.
The main result of the present paper concerns the first genuinely nontrivial case involving three blocks. 
Let \(p=12k+1\) be a prime and let \(C_i^6\) denote the \(i\)-th cyclotomic class of order six in \(\mathbb{F}_p\). We prove that the family
\[
\mathcal D=\{C_0^6,C_2^6,C_4^6\}
\]
forms an EDF if and only if \(k\) is a square. 
The proof combines symmetry relations of cyclotomic numbers with their explicit evaluations.

We also consider the two-block case from the same viewpoint.
The argument serves as a useful preliminary example for the three-block case considered later.
We also briefly discuss the difficulties arising when the number of blocks is at least four. 

The paper is organized as follows.
Section~\ref{sec2} introduces the cyclotomic setting used throughout the paper.
Section~\ref{sec3} studies the two-block case as a preliminary step toward the three-block case considered in Section~\ref{sec4}, which contains the main result of the paper.
The cyclotomic-number formulas required in these sections are summarized in Appendix~\ref{secA}.
Finally, Section~\ref{sec5} briefly comments on the difficulties of extending the method to four or more blocks.

\section{Cyclotomic Preliminaries}
\label{sec2}

In this section, we introduce the cyclotomic setting used throughout the paper. 
Let \(q=ef+1\) be a prime power, and let \(\alpha\) be a primitive element of \(\mathbb{F}_q\). 
For \(0\le i\le e-1\), the \(i\)-th \emph{cyclotomic class} of order \(e\) is defined by
\[
C_i^e=\alpha^i\langle \alpha^e\rangle
=
\{\alpha^{es+i}\mid 0\le s\le f-1\}.
\]
The classes \(C_0^e,C_1^e,\dots,C_{e-1}^e\) form a partition of \(\mathbb{F}_q^\ast\). 
For integers \(0\le i,j\le e-1\), the \emph{cyclotomic number} of order \(e\) is defined by
\[
(i,j)_e
=
|(C_i^e+1)\cap C_j^e|.
\]
We shall use the standard symmetry relations of cyclotomic numbers
(see, for example, \cite{storer1967cyclotomy}):
\begin{equation}
\label{eq:symmetry}
(i,j)_e=
\begin{cases}
(j,i)_e
& \mbox{if \(f\) is even,}\\
(j+e/2,i+e/2)_e
& \mbox{if \(f\) is odd.}
\end{cases}
\end{equation}

Let \(\mathcal D=\{D_0,D_1,\dots,D_{m-1}\}\) be a family of disjoint \(k\)-subsets of an additive abelian group \(G\) of order \(n\). 
The family \(\mathcal D\) is called an \((n,m,k,\lambda)\)-\emph{external difference family} (EDF) if every nonzero element of \(G\) appears exactly \(\lambda\) times among the external differences
\[
\bigcup_{0\le i\ne j\le m-1}\Delta(D_i,D_j),
\]
where
\[
\Delta(D_i,D_j)=\{x-y\mid x\in D_i,\ y\in D_j\}
\]
is regarded as a multiset.
By counting external differences in two ways, the parameters satisfy
\[
m(m-1)k^2=\lambda(n-1).
\]

Throughout this paper, we consider EDFs whose blocks are cyclotomic classes. More precisely, for distinct integers
\(i_0,i_1,\dots,i_{m-1}\pmod e\),
we consider the family
\[
\mathcal D
=
\{C_{i_0}^e,C_{i_1}^e,\dots,C_{i_{m-1}}^e\}.
\]
Using standard counting arguments with cyclotomic numbers
(see, for example, \cite{storer1967cyclotomy,mutoh2008difference}),
the multiplicity of an external difference depends only on the cyclotomic class containing the difference element. 
Consequently, for each \(r\), we may define \(M_r\) to be the multiplicity of an element of \(C_r^e\) among the external differences of \(\mathcal D\). 
Then \(\mathcal D\) forms an EDF if and only if the values \(M_r\) are constant for all \(r\). 
In the following sections, we investigate this condition in the cases \(m=2\) and \(m=3\).

\section{The Two-Block Case}
\label{sec3}

In this section, we study the case of two blocks. Let
\[
\mathcal D=\{C_0^4,C_2^4\}
\]
denote the family consisting of the even cyclotomic classes of order four in \(\mathbb{F}_p\).

If \(p\equiv3\pmod4\), then the classical cyclotomic construction
\(\{C_0^2,C_1^2\}\) forms a
\((p,2,(p-1)/2,(p-1)/2)\)-EDF.
When \(p\equiv1\pmod4\), however, this construction no longer yields an EDF, which naturally leads to the decomposition of the quadratic residues into cyclotomic classes of order four. In this case, the EDF condition leads to a nontrivial arithmetic characterization. 
We next consider the two-block case from the same viewpoint. 
The following argument illustrates how cyclotomic-number computations translate the EDF condition into explicit relations among arithmetic parameters.

Since \(|C_0^4|=|C_2^4|=(p-1)/4\), the family \(\mathcal D\) is a candidate for a \((p,2,(p-1)/4,(p-1)/8)\)-EDF.

\begin{lemma}
\label{lem:two-block}
Let \(p=4f+1\) be a prime, and let
\(\mathcal D=\{C_0^4,C_2^4\}\)
be the family of cyclotomic classes of order four in \(\mathbb{F}_p\).
For \(0\le r\le 3\), denote by \(M_r\) the multiplicity of an element of \(C_r^4\) among the external differences of \(\mathcal D\). Then
\[
M_r=(2-r,-r)_4+(-r,2-r)_4.
\]
\end{lemma}

\begin{proof}
Fix \(r\) with \(0\le r\le 3\), and let \(d\in C_r^4\). Consider the equation
\[
x-y=d,
\qquad
x\in C_0^4,\ y\in C_2^4.
\]
Dividing by \(d\), we obtain
\[
d^{-1}x-d^{-1}y=1.
\]
Since \(d\in C_r^4\), multiplication by \(d^{-1}\) shifts cyclotomic classes by \(-r\), so that \(d^{-1}x\in C_{-r}^4\) and \(d^{-1}y\in C_{2-r}^4\). 
Hence the number of solutions is \((2-r,-r)_4\). 
The term \((-r,2-r)_4\) is obtained similarly from
\(\Delta(C_2^4,C_0^4)\).
\end{proof}

We now derive the corresponding arithmetic characterization.

\begin{theorem}
\label{thm:two-block}
Let \(p=8k+1\) be a prime. 
Then
\(\mathcal D=\{C_0^4,C_2^4\}\)
forms a
\((p,2,(p-1)/4,(p-1)/8)\)-EDF over \(\mathbb F_p\)
if and only if \(k/2\) is a square.
\end{theorem}

\begin{proof}
By Lemma~\ref{lem:two-block}, the family \(\mathcal D\) forms an EDF
with the required parameters if and only if the values \(M_r\) are constant for all \(r\). 
The symmetry relations \eqref{eq:symmetry} then show that this condition reduces to
\[
(0,2)_4=(1,3)_4.
\]
It is well known (see, for example, \cite{cox2013primes}) that every prime
\(p\equiv 1\pmod 4\) can be represented in the form
\[
p=s^2+4t^2,
\qquad
s\equiv 1\pmod 4.
\]
Substituting the representative values of cyclotomic numbers of order four from Appendix~\ref{secA1}, we obtain
\begin{align}
8(0,2)_4 &= p-3+2s,
\label{eq:two-block-1}\\
8(1,3)_4 &= p+1-2s.
\label{eq:two-block-2}
\end{align}
Hence \((0,2)_4=(1,3)_4\) if and only if \(s=1\). Thus
\[
p=1+4t^2.
\]
Since \(p=8k+1\), we have \(t^2=2k\). Therefore \(t\) is even, say \(t=2\tau\), and hence \(k/2=\tau^2\) is a square.

Conversely, suppose that \(k/2\) is a square, say \(k=2\tau^2\). Then
\[
p=8k+1=16\tau^2+1=1+4(2\tau)^2.
\]
Thus in the representation \(p=s^2+4t^2\) with \(s\equiv 1\pmod 4\), we may take \(s=1\) and \(t=\pm 2\tau\). Substituting \(s=1\) into \eqref{eq:two-block-1} and \eqref{eq:two-block-2}, we obtain
\[
(0,2)_4=(1,3)_4.
\]
Hence \(\mathcal D\) forms a
\((p,2,(p-1)/4,(p-1)/8)\)-EDF over \(\mathbb F_p\).
\end{proof}

\begin{example}
For \(p=257=8\cdot 32+1\), we have \(k=32\) and \(k/2=4^2\). Hence
\(\mathcal D=\{C_0^4,C_2^4\}\)
forms a \((257,2,64,32)\)-EDF over  \(\mathbb{F}_{257}\) by Theorem~\ref{thm:two-block}.
\end{example}

The above argument illustrates how the same method extends naturally to the three-block case considered in Section~\ref{sec4}. 
In particular, the EDF condition is reduced to explicit equalities among cyclotomic numbers, which can then be analyzed by means of their known evaluations.

\section{The Three-Block Case}
\label{sec4}

In this section, we study the case of three blocks, which contains the main theorem of the paper. 
Let \(p=12k+1\) be a prime, and let
\[
\mathcal D=\{C_0^6,C_2^6,C_4^6\}
\]
be the family consisting of the even cyclotomic classes of order six in \(\mathbb{F}_p\). 
Since
\(|C_0^6|=|C_2^6|=|C_4^6|=(p-1)/6\),
the family \(\mathcal D\) is a candidate for a
\((p,3,(p-1)/6,(p-1)/6)\)-EDF.

\begin{lemma}
\label{lem:three-block}
Let \(p=12k+1\) be a prime, and let
\(\mathcal D=\{C_0^6,C_2^6,C_4^6\}\)
be the family of cyclotomic classes of order six in \(\mathbb{F}_p\).
For \(0\le r\le 5\), denote by \(M_r\) the multiplicity of an element of \(C_r^6\) among the external differences of \(\mathcal D\). Then
\[
M_r=
\sum_{\substack{0\le i,j\le 2\\ i\ne j}}
(2j-r,2i-r)_6.
\]
\end{lemma}

\begin{proof}
The argument is similar to that of Lemma~\ref{lem:two-block}. Fix \(r\) with \(0\le r\le 5\), and let \(d\in C_r^6\). For each ordered pair \((i,j)\) with \(i\ne j\), consider the equation
\[
x-y=d,
\qquad
x\in C_{2i}^6,\ y\in C_{2j}^6.
\]
Dividing by \(d\), we obtain
\[
d^{-1}x-d^{-1}y=1.
\]
Since multiplication by \(d^{-1}\) shifts cyclotomic classes by \(-r\), we have
\(d^{-1}x\in C_{2i-r}^6\) and
\(d^{-1}y\in C_{2j-r}^6\). Hence the number of solutions is
\[
(2j-r,2i-r)_6.
\]
Summing over all ordered pairs \((i,j)\) with \(i\ne j\) yields the formula for \(M_r\).
\end{proof}

We now obtain the main theorem of the paper.

\begin{theorem}
\label{thm:main}
Let \(p=12k+1\) be a prime. 
Then
\(\mathcal D=\{C_0^6,C_2^6,C_4^6\}\)
forms a
\((p,3,(p-1)/6,(p-1)/6)\)-EDF
over \(\mathbb{F}_p\) if and only if \(k\) is a square.
\end{theorem}

\begin{proof}
By Lemma~\ref{lem:three-block}, the family \(\mathcal D\) forms an EDF if and only if the values \(M_r\) are constant for all \(r\). From the formula for \(M_r\) and the symmetry relations in \eqref{eq:symmetry}, we have
\[
M_0=M_2=M_4
=
2\{(0,2)_6+(0,4)_6+(2,4)_6\}
\]
and
\[
M_1=M_3=M_5
=
2\{(1,2)_6+(1,3)_6+(1,4)_6\}.
\]
Thus the EDF condition is equivalent to
\[
(0,2)_6+(0,4)_6+(2,4)_6
=
(1,2)_6+(1,3)_6+(1,4)_6.
\]
It is well known (see, for example, \cite{cox2013primes}) that every prime
\(p\equiv 1\pmod 6\) can be represented in the form
\[
p=s^2+3t^2,
\qquad
s\equiv 1\pmod 3.
\]
Substituting the representative values of cyclotomic numbers of order six from Appendix~\ref{secA2}, we obtain the following identities in each of the three cases
\(2\in C_0^3\), \(2\in C_1^3\), and \(2\in C_2^3\):
\begin{align}
36\{(0,2)_6+(0,4)_6+(2,4)_6\}
&=
3p-9+6s,
\label{eq:three-even-sum}\\
36\{(1,2)_6+(1,3)_6+(1,4)_6\}
&=
3p+3-6s.
\label{eq:three-odd-sum}
\end{align}
It follows from \eqref{eq:three-even-sum} and
\eqref{eq:three-odd-sum} that the EDF condition is equivalent to \(s=1\). 
Hence \(p=1+3t^2\).
Since \(p=12k+1\), we obtain \(t^2=4k\). Thus \(t\) is even, say \(t=2u\), and hence \(k=u^2\). 

Conversely, suppose that \(k=u^2\) for some integer \(u\). Then
\[
p=12u^2+1=1+3(2u)^2
\]
is a representation of the required form with \(s=1\) and \(t=\pm 2u\). With \(s=1\), equations
\eqref{eq:three-even-sum} and
\eqref{eq:three-odd-sum} yield
\[
M_0=M_1=\cdots=M_5.
\]
Hence \(\mathcal D\) forms a
\((p,3,(p-1)/6,(p-1)/6)\)-EDF
over \(\mathbb F_p\).
\end{proof}


\begin{example}
For \(p=109=12\cdot 3^2+1\), Theorem~\ref{thm:main} implies that
\(\mathcal D=\{C_0^6,C_2^6,C_4^6\}\)
forms a \((109,3,18,18)\)-EDF over  \(\mathbb{F}_{109}\).
\end{example}

\section{Concluding Remarks}
\label{sec5}

In Sections~\ref{sec3} and~\ref{sec4}, we studied families consisting of cyclotomic classes with even indices of orders four and six.
In these cases, the EDF conditions reduce to manageable relations among cyclotomic numbers of small orders.

For constructions involving four or more blocks, however, the corresponding multiplicity formulas contain many more cyclotomic numbers, and the resulting conditions become substantially more complicated. 
Although cyclotomic numbers satisfy several symmetry relations, these relations are generally insufficient to reduce the conditions to a small number of parameters.

Another difficulty is that explicit evaluations of cyclotomic numbers become increasingly complicated for larger orders. In the order-six case, the relevant cyclotomic numbers can still be described in terms of the representation \(p=s^2+3t^2\). 
For larger orders, however, the formulas often involve additional auxiliary parameters and may depend on the cyclotomic class containing \(2\).

These difficulties indicate that the analysis becomes substantially more complicated for constructions with four or more blocks. 
Nevertheless, the results of the present paper show that cyclotomic-number techniques are effective for analyzing EDF conditions in the two-block and three-block cases.

\bmhead{Acknowledgements}

The authors would like to thank Yuto Hori for valuable discussions and contributions during the early stages of this work. 
This work was supported by JSPS KAKENHI Grant Number JP24K14819, JP24K06834, and JP22K11936 (to M. Mishima), and JP25K21151 (to Y. Tsunoda).

\begin{appendices}

\section{Cyclotomic Numbers Used in the Proofs} \label{secA}

In this appendix, we summarize the cyclotomic-number formulas used in Sections~\ref{sec3} and \ref{sec4}.

\subsection{Cyclotomic Numbers of Order Four}
\label{secA1}

The following representative values are taken from Lemma \(19'\) in Part~1 of Storer~\cite{storer1967cyclotomy}.

Let \(p=4f+1\) be a prime with \(f\) even. Then there exist integers \(s\) and \(t\) such that
\[
p=s^2+4t^2,
\qquad
s\equiv 1\pmod 4.
\]

The representative cyclotomic numbers of order four are given by
\begin{align*}
16(0,0)_4&=p-11-6s,\\
16(0,1)_4&=p-3+2s+8t,\\
16(0,2)_4&=p-3+2s,\\
16(0,3)_4&=p-3+2s-8t,\\
16(1,3)_4&=p+1-2s.
\end{align*}
The remaining values are determined by the standard symmetry relations of cyclotomic numbers.

\subsection{Cyclotomic Numbers of Order Six} \label{secA2}

We use the following representative values adapted from Theorem~9 of 
Dickson~\cite{dickson1935cyclotomy}.

Let \(p=6f+1\) be a prime with \(f\) even. Then there exist integers \(s\) and \(t\) such that
\[
p=s^2+3t^2,
\qquad
s\equiv 1\pmod 3.
\]
The explicit values depend on the cyclotomic class containing \(2\). 
These distinctions are essential in the proof of Theorem~\ref{thm:main}.

\subsubsection*{Case 1: \(2\in C_0^3\)}

The representative values are
\begin{align*}
36(0,0)_6&=p-17-20s,\\
36(0,1)_6&=p-5+4s+18t,\\
36(0,2)_6&=p-5+4s+6t,\\
36(0,3)_6&=p-5+4s,\\
36(0,4)_6&=p-5+4s-6t,\\
36(0,5)_6&=p-5+4s-18t.
\end{align*}

The remaining representative values are
\begin{align*}
36(1,2)_6
&=
36(1,3)_6
=
36(1,4)_6
=
36(2,4)_6\\
&=p+1-2s.
\end{align*}

\subsubsection*{Case 2: \(2\in C_1^3\)}

The representative values are
\begin{align*}
36(0,0)_6&=p-17-8s+6t,\\
36(0,1)_6&=p-5+4s+12t,\\
36(0,2)_6&=p-5+4s-6t,\\
36(0,3)_6&=p-5+4s-6t,\\
36(0,4)_6&=p-5-8s,\\
36(0,5)_6&=p-5+4s-6t.
\end{align*}

The remaining representative values are
\begin{align*}
36(1,2)_6
&=
36(1,3)_6
=
p+1-2s-6t,\\
36(1,4)_6&=p+1-2s+12t,\\
36(2,4)_6&=p+1+10s+6t.
\end{align*}

\subsubsection*{Case 3: \(2\in C_2^3\)}

The representative values are
\begin{align*}
36(0,0)_6&=p-17-8s-6t,\\
36(0,1)_6&=p-5+4s+6t,\\
36(0,2)_6&=p-5-8s,\\
36(0,3)_6&=p-5+4s+6t,\\
36(0,4)_6&=p-5+4s+6t,\\
36(0,5)_6&=p-5+4s-12t.
\end{align*}

The remaining representative values are
\begin{align*}
36(1,2)_6&=p+1-2s+6t,\\
36(1,3)_6&=p+1-2s-12t,\\
36(1,4)_6&=p+1-2s+6t,\\
36(2,4)_6&=p+1+10s-6t.
\end{align*}




\end{appendices}


\bibliography{edf}


\begin{thebibliography}{6}
\ifx \bisbn   \undefined \def \bisbn  #1{ISBN #1}\fi
\ifx \binits  \undefined \def \binits#1{#1}\fi
\ifx \bauthor  \undefined \def \bauthor#1{#1}\fi
\ifx \batitle  \undefined \def \batitle#1{#1}\fi
\ifx \bjtitle  \undefined \def \bjtitle#1{#1}\fi
\ifx \bvolume  \undefined \def \bvolume#1{\textbf{#1}}\fi
\ifx \byear  \undefined \def \byear#1{#1}\fi
\ifx \bissue  \undefined \def \bissue#1{#1}\fi
\ifx \bfpage  \undefined \def \bfpage#1{#1}\fi
\ifx \blpage  \undefined \def \blpage #1{#1}\fi
\ifx \burl  \undefined \def \burl#1{\textsf{#1}}\fi
\ifx \doiurl  \undefined \def \doiurl#1{\url{https://doi.org/#1}}\fi
\ifx \betal  \undefined \def \betal{\textit{et al.}}\fi
\ifx \binstitute  \undefined \def \binstitute#1{#1}\fi
\ifx \binstitutionaled  \undefined \def \binstitutionaled#1{#1}\fi
\ifx \bctitle  \undefined \def \bctitle#1{#1}\fi
\ifx \beditor  \undefined \def \beditor#1{#1}\fi
\ifx \bpublisher  \undefined \def \bpublisher#1{#1}\fi
\ifx \bbtitle  \undefined \def \bbtitle#1{#1}\fi
\ifx \bedition  \undefined \def \bedition#1{#1}\fi
\ifx \bseriesno  \undefined \def \bseriesno#1{#1}\fi
\ifx \blocation  \undefined \def \blocation#1{#1}\fi
\ifx \bsertitle  \undefined \def \bsertitle#1{#1}\fi
\ifx \bsnm \undefined \def \bsnm#1{#1}\fi
\ifx \bsuffix \undefined \def \bsuffix#1{#1}\fi
\ifx \bparticle \undefined \def \bparticle#1{#1}\fi
\ifx \barticle \undefined \def \barticle#1{#1}\fi
\bibcommenthead
\ifx \bconfdate \undefined \def \bconfdate #1{#1}\fi
\ifx \botherref \undefined \def \botherref #1{#1}\fi
\ifx \url \undefined \def \url#1{\textsf{#1}}\fi
\ifx \bchapter \undefined \def \bchapter#1{#1}\fi
\ifx \bbook \undefined \def \bbook#1{#1}\fi
\ifx \bcomment \undefined \def \bcomment#1{#1}\fi
\ifx \oauthor \undefined \def \oauthor#1{#1}\fi
\ifx \citeauthoryear \undefined \def \citeauthoryear#1{#1}\fi
\ifx \endbibitem  \undefined \def \endbibitem {}\fi
\ifx \bconflocation  \undefined \def \bconflocation#1{#1}\fi
\ifx \arxivurl  \undefined \def \arxivurl#1{\textsf{#1}}\fi
\csname PreBibitemsHook\endcsname

\bibitem[\protect\citeauthoryear{Paterson and Stinson}{2016}]{paterson2016combinatorial}
\begin{barticle}
\bauthor{\bsnm{Paterson}, \binits{M.B.}},
\bauthor{\bsnm{Stinson}, \binits{D.R.}}:
\batitle{Combinatorial characterizations of algebraic manipulation detection codes involving generalized difference families}.
\bjtitle{Discrete Mathematics}
\bvolume{339}(\bissue{12}),
\bfpage{2891}--\blpage{2906}
(\byear{2016})
\doiurl{10.1016/j.disc.2016.06.004}
\end{barticle}
\endbibitem

\bibitem[\protect\citeauthoryear{Mutoh and Tonchev}{2008}]{mutoh2008difference}
\begin{barticle}
\bauthor{\bsnm{Mutoh}, \binits{Y.}},
\bauthor{\bsnm{Tonchev}, \binits{V.D.}}:
\batitle{Difference systems of sets and cyclotomy}.
\bjtitle{Discrete Mathematics}
\bvolume{308}(\bissue{14}),
\bfpage{2959}--\blpage{2969}
(\byear{2008})
\doiurl{10.1016/j.disc.2007.08.017}
\end{barticle}
\endbibitem

\bibitem[\protect\citeauthoryear{Huczynska and Johnson}{2023}]{huczynska2023internal}
\begin{barticle}
\bauthor{\bsnm{Huczynska}, \binits{S.}},
\bauthor{\bsnm{Johnson}, \binits{L.M.}}:
\batitle{Internal and external partial difference families and cyclotomy}.
\bjtitle{Discrete Mathematics}
\bvolume{346}(\bissue{3}),
\bfpage{113295}
(\byear{2023})
\doiurl{10.1016/j.disc.2022.113295}
\end{barticle}
\endbibitem

\bibitem[\protect\citeauthoryear{Storer}{1967}]{storer1967cyclotomy}
\begin{bbook}
\bauthor{\bsnm{Storer}, \binits{T.}}:
\bbtitle{Cyclotomy and Difference Sets}.
\bsertitle{Lectures in Advanced Mathematics}.
\bpublisher{Markham Publishing Company},
\blocation{Chicago}
(\byear{1967})
\end{bbook}
\endbibitem

\bibitem[\protect\citeauthoryear{Cox}{2013}]{cox2013primes}
\begin{bbook}
\bauthor{\bsnm{Cox}, \binits{D.A.}}:
\bbtitle{Primes of the Form $x^2+ \lowercase{ny}^2$: Fermat, Class Field Theory, and Complex Multiplication},
\bedition{2}nd edn.
\bpublisher{John Wiley \& Sons},
\blocation{Hoboken, NJ}
(\byear{2013})
\end{bbook}
\endbibitem

\bibitem[\protect\citeauthoryear{Dickson}{1935}]{dickson1935cyclotomy}
\begin{barticle}
\bauthor{\bsnm{Dickson}, \binits{L.E.}}:
\batitle{Cyclotomy, higher congruences, and \uppercase{W}aring's problem}.
\bjtitle{American Journal of Mathematics}
\bvolume{57}(\bissue{2}),
\bfpage{391}--\blpage{424}
(\byear{1935})
\doiurl{10.2307/2371217}
\end{barticle}
\endbibitem

\end{thebibliography}

\end{document}